\documentclass[11pt,reqno]{amsart}%
\usepackage{graphicx}
\usepackage{amsmath}
\usepackage{amsfonts}
\usepackage{amssymb}
\usepackage{amsthm} 
\usepackage{enumerate}
\usepackage{color}
\usepackage{cite}
\usepackage{url}
\usepackage[margin=1in]{geometry}%
\usepackage{tikz-cd}
\providecommand{\U}[1]{\protect\rule{.1in}{.1in}}

\newtheorem{theorem}{Theorem}[section]
\newtheorem{proposition}[theorem]{Proposition}
\newtheorem{lemma}[theorem]{Lemma}
 
\newtheorem{definition}[theorem]{Definition}

\theoremstyle{remark}
\newtheorem{example}[theorem]{Example}
\newtheorem{remark}[theorem]{Remark}

\begin{document}
\title{Inverse eigenvalue problem for tensors}

\author{Ke Ye}
\address{Computational and Applied Mathematics Initiative, Department of Statistics,
University of Chicago, Chicago, IL 60637-1514.} 
\email{kye@math.uchicago.edu}

\author{Shenglong Hu}
\address{Department of  Mathematics, School of Science, Tianjin University, Tianjin 300072, China.}
\email{timhu@tju.edu.cn}

\begin{abstract}
Given a tensor $\mathcal T\in \mathbb T(\mathbb C^n,m+1)$, the space of tensors of order $m+1$ and dimension $n$ with complex entries, it has $nm^{n-1}$ eigenvalues (counted with algebraic multiplicities).  The inverse eigenvalue problem for tensors is a generalization of that for matrices. Namely, given a multiset $S\in \mathbb C^{nm^{n-1}}/\mathfrak{S}(nm^{n-1})$ of total multiplicity $nm^{n-1}$, is there a tensor in $\mathbb T(\mathbb C^n,m+1)$ such that the multiset of eigenvalues of $\mathcal{T}$ is exactly $S$? The solvability of the inverse eigenvalue problem for tensors is studied in this article. With tools from algebraic geometry, it is proved that the necessary and sufficient condition for this inverse problem to be generically solvable is $m=1,\ \text{or }n=2,\ \text{or }(n,m)=(3,2),\ (4,2),\ (3,3)$.
\end{abstract}
\maketitle

\section{Introduction}\label{sec:intro}

Eigenvalue of a tensor, as a natural generalized notion of the eigenvalue of a square matrix, has been attracting increasingly attention
in fields related to numerical multilinear algebra (see \cite{cs13,cpz08,fo14,lqz13,oo12,nqww07,q14} and references therein), since the independent work by Lim \cite{l05} and Qi \cite{q05}. 

For a tensor of order $m+1$ and dimension $n$, its eigenvalues are the roots of the characteristic polynomial, which is a monic polynomial with degree $nm^{n-1}$ \cite{q05,hhlq13}. As a consequence, the number of eigenvalues, counted with multiplicities, is equal to $nm^{n-1}$. This number, grows however exponentially along with $m$ and $n$, differs largely from its matrix counterpart. We can alternatively (actually equivalently) define the eigenvalues as solutions of a system of polynomial equations resembles the system of eigenvalue equations for a matrix (cf.\ Definition~\ref{def:eigenvalue-eigenvector}). However, both computation and structures of the eigenvalues are very complicated, and tough to investigate \cite{q05,hhlq13,hy15}. The situation would be improved if the eigenvalues are shown to lie in a variety in $\mathbb C^{nm^{n-1}}$ with a much smaller dimension.  


Let $\mathbb T(\mathbb C^n,m+1)$ be the space of tensors of order $m+1$ and dimension $n$ with entries in the field $\mathbb C$ of complex numbers. When $m=1$, we get the space of $n\times n$ matrices with complex components. The eigenvalues of a matrix $A=(a_{ij})\in\mathbb T(\mathbb C^n,2)$, as roots of the characteristic polynomial
\[
\operatorname{det}(\lambda I-A)=\lambda^n+c_{n-1}(A)\lambda^{n-1}+c_1(A)\lambda+c_0(A),
\]
can be written as hypergeometric series in terms of the components $a_{ij}$'s (cf.\ \cite{s02}), since $c_i(A)\in\mathbb C[A]$ is a homogeneous polynomial of degree $n-i$ for $i=1,\dots,n$. We can collect the $n$ hypergeometric series to form a multiset-valued mapping $\phi : \mathbb T(\mathbb C^n,2)\rightarrow \mathbb C^n/\mathfrak{S}(n)$, 
where $\mathfrak{S}(n)$ is the group of permutations on $n$ elements. Thus, $\phi(A)$ is the multiset of eigenvalues of $A$. The set $\mathbb C^n/\mathfrak{S}(n)$ is the $n$th symmetric product of $\mathbb C$, and (cf.\ \cite{h92})
\[
\operatorname{dim}(\mathbb C^n/\mathfrak{S}(n))=n.
\]
A well-known result from linear algebra (cf.\ \cite{hj85}) is that this mapping is surjective, i.e.,
\begin{equation}\label{matrix-surjective}
\operatorname{image}(\phi)=\mathbb C^n/\mathfrak{S}(n).
\end{equation}
This implies that any given $n$-tuple of complex numbers can be realized as the eigenvalues of an $n\times n$ matrix. 

It is shown in \cite{hhlq13} that the codegree $i$ coefficient of the characteristic polynomial of a tensor is a homogeneous polynomial of degree $i$ in terms of the tensor components. Therefore, we can define the multiset-valued eigenvalue mapping $\phi : \mathbb T(\mathbb C^n,m+1)\rightarrow\mathbb C^{nm^{n-1}}/\mathfrak{S}(nm^{n-1})$ in a similar way for the tensor eigenvalues. Here we use the same symbol $\phi$, for the sake of notational simplicity, for all positive integers $m$ and $n$, which would be clear for the content. 
Likewise, a basic question arises for tensors when one is trying to understand their eigenvalues: 
\[
\textit{what can the eigenvalues of tensors in a given space be?}
\]
This question is of course very general, hard; and deserves a very long way and much continuous effort to answer. Reversely, we can ask the question about \textit{the existence of tensors for a given multiset of eigenvalues}, which would have the nomenclature \textit{inverse eigenvalue problem for tensors} in general. We refer to \cite{f77,cg05} and references therein for the inverse eigenvalue problems for matrices. 

In this article, we will 
study the counterpart of \eqref{matrix-surjective} for tensors: when the eigenvalues fulfill the whole quotient space? It turns out that this question is hard to answer. However, with the help of concepts from algebraic geometry, we are able to answer a weaker version of the question: ``when the eigenvalues almost fulfill the whole quotient space?",  or more precisely in mathematical language: ``when does the image of the multiset-valued eigenvalue map contains an open dense subset of $\mathbb C^{nm^{n-1}}/\mathfrak{S}(nm^{n-1})$?" Unless otherwise stated, we will always adopt the Zariski topology for the ambient space. The map $\phi$ is dominant if its image 
contains an open dense subset of  $\mathbb C^{nm^{n-1}}/\mathfrak{S}(nm^{n-1})$ (cf.\  Definition~\ref{prop:dense-eig}).
We have the following main theorem of this article.
\begin{theorem}[Dominant Theorem]\label{theorem:dominant}
The eigenvalue map $\phi : \mathbb{T}(\mathbb C^n,m+1)\rightarrow\mathbb C^{nm^{n-1}}/\mathfrak{S}(nm^{n-1})$ is dominant if and only if
\[
m=1,\ \text{or }n=2,\ \text{or }(n,m)=(3,2),\ (4,2),\ (3,3).  
\]
\end{theorem}

\begin{proof}
The case $m=1$ is the trivial matrix counterpart. For the other cases, 
the necessity follows from Proposition~\ref{prop:necessary}; and the sufficiency follows from
Propositions~\ref{prop:dominant-n2}, \ref{prop:dominant-33} and \ref{prop:dominant-34-42}.
\end{proof} 
 
Since the topology on $\mathbb C^{nm^{n-1}}/\mathfrak{S}(nm^{n-1})$ is the Zariski topology, the fact that the image of $\phi$ contains an open dense subset of $\mathbb C^{nm^{n-1}}/\mathfrak{S}(nm^{n-1})$ implies that for almost all multiset $S$ in $\mathbb C^{nm^{n-1}}/\mathfrak{S}(nm^{n-1})$, there exists a tensor $\mathcal{T}$ in $\mathbb{T}(\mathbb{C}^n,m+1)$ such that the set of eigenvalues of $\mathcal{T}$ is exact $S$. More precisely, the fact that the image of $\phi$ contains an open dense subset of $\mathbb C^{nm^{n-1}}/\mathfrak{S}(nm^{n-1})$ implies that the probability that a randomly picked multiset $S$ can be realized as the set of eigenvalues of a tensor in $\mathbb{T}(\mathbb{C}^n,m+1)$ is one.

\section{Preliminaries}\label{sec:preliminary}
\subsection{Eigenvalues of tensors}\label{sec:eigenvalue}
There are two most popular definitions of tensor eigenvalues in the literature \cite{q05,l05}. Throughout this article, eigenvalues and eigenvectors of tensors are restricted to the next definition. 
\begin{definition}[Eigenvalues and Eigenvectors \cite{q05,l05}]\label{def:eigenvalue-eigenvector}
Let tensor $\mathcal T=(t_{ii_1\ldots i_m})\in\mathbb{T}(\mathbb{C}^n,m+1)$. A number $\lambda\in\mathbb C$ is called an \textit{eigenvalue} of $\mathcal T$, if there exists a vector $\mathbf x\in\mathbb C^n\setminus\{\mathbf 0\}$ which is called an \textit{eigenvector} such that
\begin{equation}\label{eigenvalue-equation}
\mathcal T\mathbf x^m=\lambda\mathbf x^{[m]},
\end{equation}
where $\mathbf x^{[m]}\in\mathbb C^n$ is an $n$-dimensional vector with its $i$-th component being $x_i^{m}$, and $\mathcal T\mathbf x^m\in\mathbb C^n$ with
\[
\big(\mathcal T\mathbf x^m\big)_i:=\sum_{i_1,\dots,i_m=1}^nt_{ii_1\dots i_m}x_{i_1}\dots x_{i_m}\ \text{for all }i=1,\dots,n.
\]
\end{definition}

A multiset $S$ of complex numbers is a $2$-tuple $(A,\psi)$ with a set $A\subseteq\mathbb C$ and a map $\psi : A\rightarrow\mathbb N_+$. For any $a\in A$, $\psi(a)$ is the \textit{multiplicity} of the element $a$ in $S$. The summation $\sum_{a\in A}\psi(a)$ is the \textit{total multiplicity} of the multiset $S$. 
The multiset of eigenvalues of a given tensor $\mathcal T$, which is denoted as $\sigma(\mathcal T)$, 
is defined as $(A,\psi)$ with $A$ being the set of eigenvalues of $\mathcal T$ and the multiplicity map $\psi$ being the algebraic multiplicity of the eigenvalue. Then, $\sigma(\mathcal T)$ is always of total multiplicity being finite \cite{q05,hhlq13},  and it is the multiset of roots of the univariate polynomial 
\[
\chi(\lambda)=\operatorname{det}(\lambda\mathcal I-\mathcal T)
\]
which is called the \textit{characteristic polynomial} of $\mathcal T$ \cite{q05}. The degree of $\chi(\lambda)$ for $\mathcal T\in\mathbb{T}(\mathbb{C}^n,m+1)$ is $nm^{n-1}$. Thus, $\sigma(\mathcal T)$ can be identified as an element in $\mathbb C^{nm^{n-1}}/\mathfrak{S}(nm^{n-1})$ for any $\mathcal T\in\mathbb T(\mathbb C^n,m+1)$. 
We refer to \cite{hhlq13,hy15} for the definitions of tensor determinant and algebraic multiplicity, and more facts on the characteristic polynomial. 

The multiset-valued eigenvalue map $\phi : \mathbb{T}(\mathbb{C}^n,m+1)\rightarrow
\mathbb C^{nm^{n-1}}/\mathfrak{S}(nm^{n-1})$  is defined as 
\[
\phi(\mathcal T)=\sigma(\mathcal T).
\]

As long as we are concerning on eigenvalues of tensors, which are solely related to $\mathcal T\mathbf x^{m}$, it is 
sufficient to consider the tensor space $\mathbb{TS}(\mathbb{C}^n,m+1):=\mathbb C^n\otimes\operatorname{S}^{m}(\mathbb C^n)$ (cf.\ \cite[Section~5.2]{hy15}). For any $\mathcal T\in\mathbb T(\mathbb C^n,m+1)$, we can symmetrize its $i$th slice  $\mathcal T_i:=(t_{ii_1\dots i_m})_{1\leq i_1,\dots,i_m\leq n}$ via
\[
\big(\mathcal T\mathbf x^{m}\big)_i=\langle\operatorname{Sym}(\mathcal T_i),\mathbf x^{\otimes (m)}\rangle:=\sum_{i_1,\dots,i_m=1}^n\big(\operatorname{Sym}(\mathcal T_i)\big)_{i_1\dots i_m}x_{i_1}\dots x_{i_m}\ \text{for all }\mathbf x\in\mathbb C^n,
\]
where $\operatorname{Sym}(\mathcal T_i)$ is the symmetrization of the tensor $\mathcal T_i$ as a symmetric tensor in the sense of the above equalities. We refer to \cite{l12} for basic concepts on tensors. Therefore, for every tensor $\mathcal T\in\mathbb T(\mathbb C^n, m+1)$, we associate it an element $\operatorname{eSym}(\mathcal T)$ in $\mathbb{TS}(\mathbb{C}^n,m+1)$ by symmetrizing its slices. It is easy to see that
\[
\mathcal T\mathbf x^{m}=\operatorname{eSym}(\mathcal T)\mathbf x^{m}\ \text{for all }\mathbf x\in\mathbb C^n.
\]
We see that all tensors in the fibre of the surjective map $\operatorname{eSym} : \mathbb T(\mathbb C^n,m+1)\rightarrow\mathbb{TS}(\mathbb C^n,m+1)$ have the same defining equations for the eigenvalue problem. Therefore, we have
\[
\phi(\mathbb{T}(\mathbb{C}^n,m+1))=\phi(\mathbb{TS}(\mathbb{C}^n,m+1)). 
\]

\subsection{Algebraic geometry}\label{sec:algebraic}
We list here some notions from algebraic geometry which we will use in this article. We refer to \cite{h77,h92,s77,clo98} for basic algebro-geometric concepts. 
\begin{enumerate}
\item An \textit{algebraic variety} in $\mathbb{C}^n$ is a set of common zeros of some polynomials in $n$ variables. In particular, the linear space $\mathbb{C}^n$ is an algebraic variety.
\item The coordinate ring $\mathbb{C}[X]$ of an algebraic variety $X$ is defined to be the quotient ring $\mathbb{C}[x_1,\cdots, x_n]/ I(X)$ where $I(X)$ is the ideal of all polynomials vanishing on $X$.
\item A map $f : X \rightarrow Y$ between two algebraic varieties $X$ and $Y$ is said to be a \textit{morphism} if $f$ is induced by a homomorphism of coordinate rings $\psi : \mathbb C[Y ] \rightarrow\mathbb C[X]$. In particular, polynomial maps between two linear spaces are morphisms. 
\item Let $f$ be a morphism between $X$ and $Y$. If its image is Zariski dense, i.e., $\overline{f(X)} = Y$ , then $f$ is called a \textit{dominant morphism}. 
\item An algebraic variety $X$ is \textit{irreducible} if $X=X_1\cup X_2$ for closed subvarieties $X_1$ and $X_2$ implies that $X_1=X$ or $X_2=X$.
\item We say that a property $P$ holds for a \textit{generic point} in $\mathbb{C}^n$ if the set of points in $\mathbb{C}^n$ that do not satisfy $P$ is contained in a proper subvariety of $\mathbb{C}^n$. For example, fix an algebraic variety $X\subset \mathbb{C}^n$, we say that a generic point in $\mathbb{C}^n$ is not in $X$.
\end{enumerate}
\begin{remark}
We remark here that if we put the outer Lebesgue measure on $\mathbb{C}^n\simeq \mathbb{R}^{2n}$ then a proper subvariety $X$ of $\mathbb{C}^n$ has measure zero. 
Hence a property $P$ holds for a generic point in $\mathbb{C}^n$ implies that the probability that a randomly picked point from $\mathbb{C}^n$ has property $P$ is one. 
\end{remark}

The main result of this article will be proved based on the following algebraic version of the open mapping theorem. 
\begin{proposition}[\cite{t02}]\label{prop:dense}
If $f:X\to Y$ is a dominant morphism between two irreducible algebraic varieties then $f(X)$ contains an open dense subset of $Y$.
\end{proposition}

The following two facts are obvious to those who are familiar with algebraic geometry, while we supply proofs here for completeness. 
\begin{proposition}\label{prop:necessary-general}
If a morphism $f : X \rightarrow Y$ between two algebraic varieties $X$ and $Y$ is dominant, then it holds that
\[
\operatorname{dim}(X)\geq \operatorname{dim}(Y).
\]
\end{proposition}

\begin{proof}
It is known that $\operatorname{dim}(X)$ is the same as the transcendence degree of the function field $\mathbb{C}(X)$ over $\mathbb{C}$ and $f$ is dominant if and only if the ring map $\psi: \mathbb{C}[Y]\to \mathbb{C}[X]$ is an inclusion of rings. Since $\psi$ is an inclusion of rings we obtain that $\psi$ induces an inclusion of fields $\mathbb{C}(Y)\to \mathbb{C}(X)$. Therefore we have 
\[
\operatorname{tr.d}_{\mathbb{C}}(\mathbb{C}(X))\ge \operatorname{tr.d}_{\mathbb{C}}(\mathbb{C}(Y)).
\]
\end{proof}

An algebraic variety $X\subseteq\mathbb C^n$ is \textit{smooth} if the tangent space $T_{\mathbf x}(X)$ has constant dimension (i.e., $\operatorname{dim}(X))$ for every $\mathbf x\in X$. 
\begin{proposition}\label{prop:differential}
Let $f : X \rightarrow Y$ be a morphism between two smooth algebraic varieties with $\operatorname{dim}(X)\geq \operatorname{dim}(Y)$.
If there exists a point $\mathbf x\in X$ such that the rank of the differential of $f$ at $\mathbf x$ is equal to $\operatorname{dim}(Y)$, then the morphism $f$ is dominant. 
\end{proposition}
\begin{proof}
If $f$ is not dominant then $\overline{f(X)}$ is a proper subvariety of $Y$. Hence it factors as 
\[
X\xrightarrow{g} f(X)\xrightarrow{i} Y,
\]
where $g$ is defined by $g(\mathbf x)=f(\mathbf x)$ and $i$ is the inclusion of $\overline{f(X)}$ into $Y$. Then the differential $d_{\mathbf x}f$  factors as 
\[
T_{\mathbf x}X \xrightarrow{d_{\mathbf x}g} T_{f(\mathbf x)} \overline{f(X)} \xrightarrow{d_{f(\mathbf x)}i}T_{f(\mathbf x)}Y,
\]
where $T_{\mathbf x}X$ is the tangent space of the variety $X$ at the point $\mathbf x$.
Since $\overline{f(X)}$ is a proper subvariety of $Y$, it has strictly smaller dimension than $\operatorname{dim}(Y)$, which implies that $\operatorname{rank}(d_{\mathbf x} g)$ is at most $\operatorname{dim} T_{f(x)}\overline{f(X)} < \operatorname{dim}(Y)$. Therefore, we get a contradiction to the assumption that the rank of $d_{\mathbf x}f$ is $\operatorname{dim}(Y)$.
\end{proof}

For any positive integer $d>0$, the roots of the univariate polynomial equation
\[
t^d+p_{d-1}t^{d-1}+\dots+p_1t+p_0=0
\]
depends continuously on the coefficient vector $\mathbf p:=(p_{d-1},\dots,p_0)^\mathsf{T}$ \cite{s02}. If we collect the $d$ trajectories of the roots, we can define a multiset-valued map $q :\mathbb C^d\rightarrow\mathbb C^d/\mathfrak{S}(d)$ by
\begin{equation}\label{map:root}
q(\mathbf w):= \{ \text{roots (with mutliplicities) of  }t^d+w_1t^{d-1}+\dots+w_d=0\}.
\end{equation}

\begin{definition}\label{prop:dense-eig}
Let 
$p_i(\mathbf y)\in\mathbb C[\mathbf y]$ be a polynomial for all $i=0,\dots,d-1$ with $\mathbf y=(y_1,\dots,y_k)^\mathsf{T}\in\mathbb C^k$, and $\mathbf p: \mathbb C^k\rightarrow\mathbb C^d$ be the mapping defined by:
\[
\mathbf p(\mathbf y):=(p_{d-1}(\mathbf y),\dots,p_0(\mathbf y))^\mathsf{T}.
\]
The mapping $q\circ\mathbf p : \mathbb C^k\rightarrow\mathbb C^d/\mathfrak{S}(d)$ is called dominant, if $\operatorname{image}(q\circ\mathbf p)$ contains a Zariski open dense subset of $\mathbb C^d/\mathfrak{S}(d)$.
\end{definition}
Definition~\ref{prop:dense-eig} is an extension of dominant morphisms, since $q\circ\mathbf p$ is not a morphism.

\begin{lemma}\label{lemma:roots}
For any positive integer $d>0$,
let $\mathbf p: \mathbb C^k\rightarrow\mathbb C^d$ be a polynomial mapping as in Definition~\ref{prop:dense-eig}. Then, the composite mapping $q\circ\mathbf p : \mathbb C^k\rightarrow\mathbb C^d/\mathfrak{S}(d)$ is surjective (respectively, dominant), i.e., 
\[
\operatorname{image}(q\circ \mathbf p) =\mathbb C^d/\mathfrak{S}(d)  (\text{respectively, }
\operatorname{image}(q\circ \mathbf p)\ \text{contains an open dense subset of }\mathbb C^d/\mathfrak{S}(d))
\]
if and only if the mapping $\mathbf p$ is a surjective (respectively, dominant) morphism, i.e., 
\[
\operatorname{image}(\mathbf p)=\mathbb C^d (\text{respectively, }
\overline{\operatorname{image}(\mathbf p)}=\mathbb C^d).
\]
\end{lemma}

\begin{proof}
Note that $q : \mathbb C^d\rightarrow\mathbb C^d/\mathfrak{S}(d)$ is bijective. Then, $q\circ\mathbf p$ is surjective if and only if $\mathbf p$ is surjective.

We consider the map $g : \mathbb C^{d}/\mathfrak{S}(d)\to \mathbb{C}^{d}$ defined by sending a multiset $\{\lambda_1,\dots, \lambda_{d}\}$ to the vector formed by coefficients (except the leading term) of the polynomial $(t-\lambda_1)\cdots (t-\lambda_{d})$ in increasing codegree order. Then $g$ is a morphism. It is easy to see that $g$ is the inverse of the map $q :  \mathbb{C}^{d}\to\mathbb C^{d}/\mathfrak{S}(d)$. 

If $q\circ\mathbf p$ is dominant, then there is $V\subseteq \operatorname{image}(q\circ\mathbf p)$ such that $V$ is an open dense subset of $\mathbb C^{d}/\mathfrak{S}(d)$. $V$ is also an Euclidean open subset. 
Since $q$ is in addition continuous, $q^{-1}(V)=g(V)\subseteq \mathbf p(\mathbb C^k)$
is an Euclidean open subset. If $q^{-1}(V)$ is not dense, then there is small Euclidean open ball $\hat V\subset\mathbb C^d$ such that $q^{-1}(V)\cap \hat V=\emptyset$. Since $g$ is continuous, $g^{-1}(\hat V)$ is an Euclidean open set in $\mathbb C^{d}/\mathfrak{S}(d)$ (whose Euclidean topology is the induced one from $\mathbb C^d$). We must have $g^{-1}(\hat V)\cap V=\emptyset$, since $g$ is bijective.  
Thus, we obtain a contradiction to the choice of $V$. Therefore, $\mathbf p(\mathbb C^k)$ should contains an Euclidean open dense subset of $\mathbb C^d$. On the other hand, it is also true that the Euclidean closure of $\mathbf p(\mathbb C^k)$ is contained in the Zariski closure of $\mathbf p(\mathbb C^k)$. Thus, $\overline{\mathbf p(\mathbb C^k)}=\mathbb C^d$, and hence $\mathbf p$ is a dominant morphism. 

Suppose that $\mathbf p : \mathbb C^k\to\mathbb C^d$ is a dominant morphism. It follows from Proposition~\ref{prop:dense} that $\mathbf p(\mathbb C^k)$ contains an open dense subset $U$ of $\mathbb C^d$. Since $g$ is a morphism, $g^{-1}(U)$ is an open dense subset of $\mathbb C^d/\mathfrak{S}(d)$. Therefore, $g^{-1}(U)=q(U)\subseteq q(\mathbf p(\mathbb  C^k))\subseteq \mathbb C^d/\mathfrak{S}(d)$ is an open dense subset.  Thus, $q\circ\mathbf p$ is dominant by Definition~\ref{prop:dense-eig}. 
\end{proof}

\subsection{Necessary conditions}\label{sec:necessary}
We can expand out the characteristic polynomial $\chi(\lambda)$ of a tensor $\mathcal T\in\mathbb {TS}(\mathbb C^n,m+1)$ as
\[
\chi(\lambda)=\lambda^{nm^{n-1}}+c_{nm^{n-1}-1}(\mathcal T)\lambda^{nm^{n-1}-1}+\dots+c_1(\mathcal T)\lambda+c_0(\mathcal T).
\]
According to \cite{hhlq13}, each $c_i(\mathcal T)\in\mathbb C[\mathcal T]$ is a homogeneous polynomial in the variables $t_{ii_1\dots i_m}$'s of degree $nm^{n-1}-i$ for $i=0,\dots,nm^{n-1}-1$. We define the \textit{coefficient map} $\mathbf c : \mathbb{TS}(\mathbb C^n,m+1)\rightarrow \mathbb C^{nm^{n-1}}$ as
\begin{equation}\label{coefficient-map}
\mathbf c(\mathcal T):=(c_{nm^{n-1}-1}(\mathcal T),\dots,c_0(\mathcal T))^\mathsf{T}\ \text{for all }\mathcal T\in \mathbb{TS}(\mathbb C^n,m+1).
\end{equation}
It is easy to see that $\mathbf c$ is a morphism between two smooth varieties. It is also easy to see that $\phi=q\circ\mathbf c$ (cf.\ Definition~\ref{prop:dense-eig}).
These, together with Lemma~\ref{lemma:roots}, imply the next proposition. 

\begin{proposition}[Equivalent Relation]\label{prop:tensor}
For any positive integers $m$ and $n$, the multiset-valued eigenvalue map 
$\phi: \mathbb {TS}(\mathbb C^n,m+1) \to \mathbb{C}^{nm^{n-1}}/\mathfrak{S}_{nm^{n-1}}$
is surjective (respectively, dominant)  if and only if the coefficient map $\mathbf c : \mathbb {TS}(\mathbb C^n,m+1) \to\mathbb C^{nm^{n-1}}$ is a surjective (respectively, dominant) morphism. 
\end{proposition}

\begin{lemma}\label{lemma:size}
For all positive integers $m,n\geq 2$, it holds that
\begin{equation}\label{lemma:size-inequality}
{n+m-1\choose m}<m^{n-1},
\end{equation}
unless $n=2$, or
\[
(n,m)=(3,2),\ (4,2),\ (3,3).  
\]
\end{lemma}

\begin{proof}
First note that, for fixed $m\geq 2$, if \eqref{lemma:size-inequality} hods for some $n\geq 2$, then it also holds for
$n+1$, since
\[
{n+m\choose m}=\frac{n+m}{n}{n+m-1\choose m}<m{n+m-1\choose m}.
\]
Second, note that for fixed $n\geq 2$, if \eqref{lemma:size-inequality} hods for some $m\geq 2$, then it also holds for
$m+1$, since
\[
\frac{{n+m-1\choose m}}{m^{n-1}}=\frac{(1+\frac{n-1}{m})\cdots(1+\frac{1}{m})}{(n-1)!}.
\]
Last, it is then a direct calculation to see that the listed cases are the only exceptions to the inequality~\eqref{lemma:size-inequality}. 
\end{proof}

The next proposition establishes the necessary condition under which the multiset-valued eigenvalue map is dominant. It says that in most situations, the eigenvalue map $\phi$ fails to be dominant. 
\begin{proposition}[Necessary condition]\label{prop:necessary}
Let integers $m,n\geq 2$. 
A necessary condition for the map $\phi : \mathbb{TS}(\mathbb C^n,m+1)\rightarrow\mathbb C^{nm^{n-1}}/\mathfrak{S}(n)$ being dominant 
is that either 
$n=2$, or 
\[
(n,m)=(3,2),\ (4,2),\ (3,3).  
\]
\end{proposition}

\begin{proof}
Note that
the dimension of the tensor space 
$\mathbb{TS}(\mathbb{C}^n,m+1)$ is
\[
n{n+m-1\choose m}.
\]
The result then follows from Propositions~\ref{prop:necessary-general} and \ref{prop:tensor}, and Lemma~\ref{lemma:size}. 
\end{proof}

\section{Tensors with dimension $n=2$}\label{sec:two}

\subsection{Basics}
In this section, we consider tensors in $\mathbb {TS}(\mathbb C^2,m+1)$. The multiset-valued eigenvalue map is therefore 
\[
\phi : \mathbb{TS}(\mathbb C^2,m+1)\rightarrow\mathbb C^{2m}/\mathfrak{S}(2m).
\]
The system of eigenvalue equations of a tensor $\mathcal T=(t_{i_0\dots i_m})$ is (cf.\ \eqref{eigenvalue-equation})
\[
\begin{cases}a_0x^m+a_1x^{m-1}y+\dots+a_my^m=\lambda x^m,\\ b_0x^m+\dots+b_{m-1}xy^{m-1}+b_my^m=\lambda y^m,\end{cases}
\]
where we parameterized $\mathcal T$ as
\[
a_0:=t_{111\dots 1},\ a_1:=t_{121\dots 1}/m,\dots, a_m=t_{122\dots2},\ b_0=t_{211\dots1},\dots,b_{m-1}=t_{212\dots 2}/m,\ b_m=t_{222\dots 2}.
\]

It follows from the Sylvester formula for the resultant of two homogeneous polynomials in two variables (cf.\ \cite{s02,gkz94}) that 
the characteristic polynomial is $\operatorname{det}(M-\lambda I)$ with the identity matrix $I\in\mathbb C^{2m\times 2m}$ and
the matrix $M\in\mathbb C^{2m\times 2m}$
\begin{equation}\label{sylvester-matrix}
M=\begin{bmatrix}a_0&a_1&a_2&\dots&a_m&0&0&\dots\\ 0& a_0&a_1&a_2&\dots&a_m&0&\dots\\
0&0& a_0&a_1&a_2&\dots&a_m&\dots \\ &&&\dots\\ 0&\dots &0&a_0&a_1&a_2&\dots& a_m\\
b_0&b_1&b_2&\dots&b_m&0&0&\dots\\ 0& b_0&b_1&b_2&\dots&b_m&0&\dots\\
0&0& b_0&b_1&b_2&\dots&b_m&\dots \\ &&&\dots\\ 0&\dots &0&b_0&b_1&b_2&\dots& b_m
\end{bmatrix}.
\end{equation}

For all $k=1,\dots,2m$, denote by $M_k:=\{A : A\ \text{is a }k\times k\ \text{principal submatrix of }M\}$ the set of all $k\times k$ principal submatrices of $M$. 
It is known that
\[
\operatorname{det}(M-\lambda I)=\sum_{k=0}^{2m}(-1)^k\bigg(\sum_{A\in M_{2m-k}}\operatorname{det}(A)\bigg)\lambda^k,
\]
where $M_0:=\emptyset$ by convenience, and the summation over an empty set is defined as $1$. 
Denote by
\begin{equation}\label{minor}
c_k(\mathcal T):=(-1)^k\sum_{A\in M_{2m-k}}\operatorname{det}(A), \ \text{for all }k=0,\dots,2m.
\end{equation}
We have (cf.\ \cite{hhlq13})
\[
c_0(\mathcal T)=\operatorname{det}(\mathcal T)=\operatorname{det}(M),\ c_{2m-1}(\mathcal T)=-m(a_0+b_m),\ \text{and }c_{2m}(\mathcal T)=1.
\]
It is easy to see that each $c_i\in\mathbb C[\mathcal T]$ is a homogeneous polynomial of degree $2m-i$ for $i=0,\dots,2m$, and
$c_{2m-1}(\mathcal T),\dots,c_0(\mathcal T)$ are the components of the coefficient map $\mathbf c$ (cf.\ \eqref{coefficient-map}). 
Denote by $H\in\mathbb C^{2m\times(2m+2)}$ the Jacobian matrix of the coefficient map $\mathbf c:=(c_{2m-1},\dots,c_0)^\mathsf{T}:\mathbb C^{2m+2}\rightarrow\mathbb C^{2m}$ with respect to variables $a_0,\dots,a_m,b_0,\dots,b_m$:
\[
h_{ij}:=\begin{cases}\frac{\partial c_{2m-i}}{\partial a_{j-1}}& \text{if }j\leq m+1,\\ \frac{\partial c_{2m-i}}{\partial b_{j-m-2}}& \text{otherwise}.\end{cases}
\]
Denote by the submatrix $H_{:,1:2m}$ of $H$ as $K$.  
Here we use the Matlab notation for submatrices: $A_{a:b,c:d}$ means the submatrix of $A\in\mathbb C^{p\times q}$ formed by the row index set $\{a,a+1,\dots,b\}$ and the column index set $\{c,c+1,\dots,d\}$, $A_{:,c:d}$ means the corresponding row index set being the entire  $\{1,\dots,p\}$, etc. So,  
$K$ is a $2m\times 2m$ matrix with entries in $\mathbb C[a_0,\dots,a_m,b_0,\dots,b_m]$. Moreover, it follows that the monomial of every term in each entry of the $i$th row of $K$ is of the same degree $i-1$ with respect to $a_0,\dots,a_m,b_0,\dots,b_m$. 

In order to show that the map $\phi$ is dominant for $\mathbb{TS}(\mathbb C^2,m+1)$, which is the same as the map $\mathbf c$ being dominant (cf.\ Proposition~\ref{prop:tensor}), 
our goal is to show that the matrix $H$ is of full rank for some tensor $\mathcal T$ (cf.\ Proposition~\ref{prop:differential}), which will be a consequence of the nonsingularity of 
$K$ at that tensor point. Actually, we will show a much more stronger result: the determinant of the matrix $K$ is a nonzero polynomial
in $\mathbb C[\mathcal T]$, which implies the nonsingularity generically. To achieve this, we only need to show that there is a term $\alpha a_1^{\frac{m(m-1)}{2}}a_m^{m-1}b_{m-1}^{\frac{m(m+1)}{2}+(m-1)^2}$ in the determinant $\operatorname{det}(K)$ for some nonzero scalar $\alpha$. 


To illustrate the proof of the general case we first work out the following example. 
\begin{example}
Let $m=2$. Then we have the Sylvester matrix 
\[
M=\left[
\begin{matrix}
a_0 & a_1 & a_2 & 0\\
0 & a_0 & a_1 & a_2\\
b_0 & b_1 & b_2 & 0\\
0 & b_0 & b_1 & b_2 
\end{matrix}
\right]
\]
and hence the coefficients of the characteristic polynomial of $M$ (the same as that for the tensor) are
\begin{align*}
c_4(M)&=1,\\
c_3(M)&=-2(a_0+b_2),\\
c_2(M)&=\operatorname{det}\left[
\begin{matrix}
a_0 & a_1\\
b_1 & b_2
\end{matrix}
\right]+ \textit{ principal $2\times 2$ minors which do not involve $b_1a_i$'s}, \\
c_1(M)&=-\operatorname{det} \left[
\begin{matrix}
a_0 & a_1 & a_2\\
b_1 & b_2 & 0\\
b_0 & b_1 & b_2
\end{matrix}
\right] - \textit{ principal $3\times 3$ minors which do not involve  $b_1^2a_i$'s},\\
c_0(M)& = \operatorname{det}(M), \textit{in which only one term can involve $b_0b_1$, that is, $a_1a_2b_0b_1$.}
\end{align*}
It is easy to compute $H$
\[
H=
\begin{bmatrix}
-2 & 0 & 0 & 0 & 0 & 2\\
* & -b_1 & * & * & \% & \% \\
\& &\&  & -b_1^2+ \& & \& & \% & \% \\
\$ & \# &\# &  -a_1a_2b_1+\# & \% & \%
\end{bmatrix},
\]
where $*$'s contain terms without the variable $b_1$, 
$ \&$'s contain terms with the degrees of $b_1$ being strictly smaller than $2$,
and $\#$ contains terms violating either (1) has the variable $b_1$  or (2) only has the variables $a_1$, $a_2$ and $b_1$. 

By definition the submatrix $K$ of $H$ is 
\[
K=
\begin{bmatrix}
-2 & 0 & 0 & 0  \\
* & -b_1 & * & *  \\
\& &\&  & -b_1^2+ \& & \&  \\
\$ & \# &\# &  -a_1a_2b_1+\#  
\end{bmatrix}.
\]
Thus, the only way to obtain $a_1a_2b_1^4$ in $\operatorname{det}(K)$ is by taking the diagonal entries of $K$. It is obvious that the coefficient of $a_1a_2b_1^4$ is nonzero. Note that the degree $4$ for the variable $b_1$ is the maximal possible.  
\end{example} 

\subsection{General cases}
Let us look at the diagonal elements of the submatrix $K_{1:m+1;1:m+1}$ of $K$. 
\begin{lemma}\label{lemma:diagonal}
For each $i=1,\dots,m+1$, this is a nonzero term of the monomial $b_{m-1}^{i-1}$ in the entry $K_{ii}$. 
Moreover, this is the unique entry in the $i$th row of $K$ containing a term of the monomial $b_{m-1}^{i-1}$. 
\end{lemma}

\begin{proof}
Let us visualize the submatrix $M_{m-1:2m,m-1:2m}$ of $M$:
\[
P=\begin{bmatrix}a_0&a_1&a_2&a_3&\dots&a_m\\ b_{m-1}&b_m&0&0&\dots&0\\ &b_{m-1}&b_m& \\ & &\ddots&\ddots&\\
&&&b_{m-1}&b_m&0\\ &&&&b_{m-1}&b_m
\end{bmatrix}.
\]
The case when $i=1$ is trivial. Let $i>1$. 
It is easy to check that there is a term
\[
(-1)^{i-1}a_{i-1}b_{m-1}^{i-1}
\]
in the $i\times i$ leading principal minor of $P$ for all $i=2,\dots,m+1$. 
It is also easy to see that any other choice of $i\times i$ principal minor of $M$ cannot have a term whose monomial is $a_{i-1}b_{m-1}^{i-1}$ for $i>1$, since only $i\times i$ principal submatrices of $P$ can contain $i-1$ rows for the variable $b_{m-1}$ and one row for $a$'s, and only the minor we have seen can result in a nonzero term of the monomial $a_{i-1}b_{m-1}^{i-1}$. 
Henceforth, it follows from the formulae for the coefficients and the definition for the Jacobian matrix that a monomial $b_{m-1}^{i-1}$ occurs in the entry $K_{ii}$ for all $i=1,\dots,m+1$.

In the next, we show the uniqueness. By the homogeneity of the polynomials in each entry, the case $i=1$ is trivial. Actually, it follows from \cite{hhlq13} that $c_{2m-1}(\mathcal T)=-m(a_0+b_m)$, which implies $K_{1j}=0$ for $j=2,\dots,2m$. 

Let us fix $i>1$. First, each entry $K_{ij}$ cannot have a nonzero term of the monomial $b_{m-1}^{i-1}$ for $j>m+1$. Suppose on the contrary that it has, then $c_{2m-i}$ contains a nonzero term of the monomial $b_{j-m-2}b_{m-1}^{i-1}$. It follows from the structure of the matrix $M$ that this term comes from an $i\times i$ principal minor of the submatrix $M_{m+1:2m,m+1:2m}$:
\[
M_1:=\begin{bmatrix}b_m&0&0&\dots&0\\ b_{m-1}&b_m&0&\dots&0\\ b_{m-2}&b_{m-1}&b_m&\dots&0\\
\dots&\dots&\dots&\dots&\dots\\
b_1&b_2&b_3&\dots&b_m
\end{bmatrix}.
\]
However, this cannot happen, since the monomial of every term in any principal minor of this matrix contains the variable $b_m$. 

Second, we show that each entry $K_{ij}$ cannot have a nonzero term of the monomial $b_{m-1}^{i-1}$ for $j\neq i$ within $j\in\{1,\dots,m+1\}$.
Again, suppose on the contrary that it has, then $c_{2m-i}$ contains a nonzero term of the monomial $a_{j-1}b_{m-1}^{i-1}$. It comes from a principal minor of $M$. The corresponding principal submatrix is denoted by $T\in\mathbb C^{i\times i}$. By the hypothesis, we should have that $T$ is such a principal submatrix with whose $(i-1)\times (i-1)$ lasting principal submatrix comes from an $(i-1)\times (i-1)$ principal submatrix of $M_1$, since we should have $i-1$'s $b_{m-1}$. 
Note that each principal submatrix of $M_1$ is lower triangular with the last diagonal entry being $b_m$. Therefore, in order to get the monomial $a_{j-1}b_{m-1}^{i-1}$, we must have that the $(1,i)$th entry of $T$ being $a_{j-1}$, and there is $b_{m-1}$ in each $s$th row of $T$ for $s=2,\dots,i$ by Laplace's determinant formula (cf.\ \cite{hj85}). However, this can only happen when $j=i$ and $T$ being a leading principal submatrix of $P$. A contradiction is therefore arrived. 

In conclusion, $K_{ii}$ is the unique entry in the $i$th row of $K$ possessing a nonzero term of the monomial $b_{m-1}^{i-1}$.  
\end{proof}

Let us look at the antidiagonal elements of the submatrix $K_{m+2:2m;m+2:2m}$ of $K$. 
\begin{lemma}\label{lemma:antidiagonal}
For each $i=m+2,\dots,2m$, there is a nonzero term of the monomial $a_1^{i-m-1}b_{m-1}^{m-1}a_m$ in the $(i,3m-i+2)$th entry of $K$. 
Moreover, it is the unique entry with the maximal degree $m-1$ for $b_{m-1}$ among terms involving only $a_1,b_{m-1},a_m$ in the $(i,j)$th entry of $K$ for $j=2,\dots,2m$.
\end{lemma}

\begin{proof}
Obviously, we cannot get a nonzero term with degree $m$ for the variable $b_{m-1}$ in the $(i,j)$th entry of $K$ for all
$j=m+2,\dots,2m$ and $i=m+2,\dots,2m$, since only $m$ rows of $M$ contain $b$'s.  
It can be seen from the submatrix $M_{2m-i+1:2m,2m-i+1:2m}$ of $M$ that 
a nonzero term of the monomial
\[
b_{2m-i}a_1^{i-m-1}b_{m-1}^{m-1}a_m
\]
occurs in its determinant, which is an $i\times i$ principal minor. We claim that the determinant of $M_{2m-i+1:2m,2m-i+1:2m}$ is the unique $i\times i$ minor in the definition of  $c_{2m-i}$ (cf.\ \eqref{minor}) which has a nonzero term of the monomial $b_{2m-i}a_1^{i-m-1}b_{m-1}^{m-1}a_m$. 
To obtain $b_{2m-i}b_{m-1}^{m-1}$, for any $i\times i$ principal submatrix $T$ of $M$, the matrix $P$ (defined in the proof of Lemma~\ref{lemma:diagonal}) should be its principal submatrix as well, since only $m$ rows of $M$ contain $b$'s and we should take them all.  

First, the last column of the matrix $T$ contains only $a_m$, $b_m$ and $0$'s, from which $a_m$ should be chosen, since $2m-i<m$ for all the possible $i=m+2,\dots,2m$. 
Second, we cannot choose $b_{2m-i}$ from the lower triangular parts of the submatrix $P$, since otherwise, we can at most get 
$b_{m-1}^{m-2}$ according to Laplace's determinant formula. 
Third, by the second, we can only choose $b_{2m-i}$ from the first $i-m-1$ columns of $T$. Also, since we pick principal submatrices from $M$, the $(1,1)$th entry of $T$ would be $a_0$ for sure, and the others in the first column are distinct $b_t$'s. Therefore, we must choose $b_{2m-i}$ from the first column. Moreover, it should be the first nonzero entry other than $a_0$ in the first column.  
It then follows from the structure of the matrix $M$ that the only possible principal submatrix is the submatrix $M_{2m-i+1:2m,2m-i+1:2m}$. 

Therefore, it follows from the formulae for the coefficients and the definition for the Jacobian matrix that a nonzero term of the monomial $a_1^{i-m-1}b_{m-1}^{m-1}a_m$ occurs in the $(i,3m-i+2)$th entry of $K$ for all $i=m+2,\dots,2m$. 

With almost the same argument, we can see that there does not exist a nonzero term of a monomial with the maximum degree $m-1$ for $b_{m-1}$ and with only the variables $a_1$, $b_{m-1}$ and $a_m$ in the $(i,j)$th entry of $K$ for all $j\in\{m+2,\dots,2m\}\setminus\{3m-i+2\}$ and $i=m+2,\dots,2m$. 

In the next, 
we show that there does not exist a nonzero term of a monomial
\[
a_ra_1^pb_{m-1}^{m-1}a_m^q
\]
with some integers $p+q=i-m$ for any $r=1,\dots,m$ in any $i\times i$ principal minor of $M$ for all $i=m+2,\dots,2m$.
Note that the first column of any $i\times i$ principal submatrix of $M$ is of the form 
\[
(a_0,0,\dots,0,b_t,b_{t-1},\dots)^\mathsf{T} 
\]
for some $t<m-1$, since $i\geq m+2$. Therefore, each term of the minor must contain either the variable $a_0$ or a variable $b_s$ for some $s<m-1$. Neither case will result in a nonzero term of the monomial involving only $a_1$, $a_r$, $a_m$ and $b_{m-1}$ for some $r=1,\dots,m$. Thus, no term involving only $a_1$, $a_m$ and $b_{m-1}$ exists in the $(i,j)$-entry of $K$ for $i=m+2,\dots,2m$ and $j=2,\dots,m+1$. 

Henceforth, a nonzero term of the monomial $a_1^{i-m-1}b_{m-1}^{m-1}a_m$ uniquely appears in the $(i,3m-i+2)$-entry of $K$ for every $i=m+2,\dots,2m$. 
\end{proof}

\begin{lemma}\label{lemma:nonsingular}
The submatrix $K$ of the Jacobian matrix $H$ is nonsingular generically over the tensor space $\mathbb T(\mathbb C^2,m+1)$ for all $m=2,3,\dots$.
\end{lemma}

\begin{proof}
We know that the first row of $K$ is
\[
(-m,0,\dots,0)^\mathsf{T}\in\mathbb C^{2m}, 
\]
which implies that $\operatorname{det}(K)=-m\operatorname{det}(K_{2:2m,2:2m})$. 

We consider terms of $\operatorname{det}(K)$ of monomials with only the variables $a_1$, $b_{m-1}$ and $a_m$. 
For $i=2,\dots,2m$, each entry of the $i$th row of the matrix $K$ is a homogeneous polynomial of degree $i-1$ in the variables $a_0,\dots,a_m,b_0,\dots,b_m$, and there are $m$ rows of $M$ consisting $b_{m-1}$. 
This, together with Lemmas~\ref{lemma:diagonal} and \ref{lemma:antidiagonal}, implies that the maximal possible degree for the variable $b_{m-1}$ in such a term in the determinant of the matrix $K$ is
\[
1+\dots+m+(m-1)(m-1)=\frac{m(m+1)}{2}+(m-1)^2.
\]
It follows from Lemmas~\ref{lemma:diagonal} and \ref{lemma:antidiagonal} again that such a term is unique and there is a unique way to consititute it:
choosing the diagonal entries of the submatrix  $K_{1:m+1,1:m+1}$ and then the anti-diagonal entries of the submatrix $K_{m+2:2m,m+2:2m}$. Moreover, by the same lemmas, 
the term of the monomial $a_1^{\frac{m(m-1)}{2}}b_{m-1}^{\frac{2m^2-m+1}{2}}a_m^{m-1}$ in the determinant of the $K$ has a nonzero coefficient. 

Therefore, the determinant of the matrix $K$ is a nonzero polynomial over the polynomial ring $\mathbb C[a_0,\dots,a_m,b_0,\dots,b_m]$. By Hilbert's zero theorem (cf. \cite{h77,h92,s02}), we conclude that the submatrix $K$ of the Jacobian matrix is nonsingular generically in the tensor space. 
\end{proof}

\begin{proposition}\label{prop:dominant-n2}
For any positive $m\geq 1$, the multiset-valued eigenvalue map $\phi : \mathbb{TS}(\mathbb C^2,m+1)\rightarrow\mathbb C^{2m}/\mathfrak{S}(2m)$ is dominant, i.e.,  
for a generic multiset $S\in\mathbb C^{2m}/\mathfrak{S}(2m)$, there exists a tensor $\mathcal T\in\mathbb{TS}(\mathbb C^2,m+1)$ such that
the set of eigenvalues (counting with multiplicities) of $\mathcal T$ is $S$. 
\end{proposition}

\begin{proof}
It follows from Propositions~\ref{prop:differential} and \ref{prop:tensor}, and Lemma~\ref{lemma:nonsingular}. 
\end{proof}

\subsection{Extensions}\label{sec:ext}
We first make a parenthesis on Sylvester matrices. 
A Sylvester matrix is a matrix of the form as $M$ (cf.\ \eqref{sylvester-matrix}): 
\begin{equation*}
\begin{bmatrix}a_0&\dots&a_p&0&0&\dots\\  &&&\dots\\ 0&\dots &0&a_0&\dots& a_p\\
b_0&\dots&b_q&0&0&\dots\\  &&&\dots\\ 0&\dots &0&b_0&\dots& b_q
\end{bmatrix},
\end{equation*}
while in general there are $q$ rows of $a$'s and $p$ row of $b$'s for different $p,q$. Therefore, the matrix is in $\mathbb C^{(p+q)\times (p+q)}$.
Up to permutation, we can assume without loss of generality that $q\geq p$. Then, with almost the same argument as 
the preceding analysis, we can obtain the following result on the inverse eigenvalue problem for Sylvester matrices.
\begin{proposition}[Sylvester Matrix]\label{prop:sylvester}
Let $m\geq 2$ be a positive integer. Given a generic multiset $S\in\mathbb C^{m}/\mathfrak{S}(m)$ there exists a Sylvester matrix $A\in\mathbb C^{m\times m}$ such that the set of eigenvalues (counting with multiplicities) of $A$ is $S$. 
\end{proposition}

In the following, we will get back to tensors. 
Note that we have a decomposition of $\mathbb{TS}(\mathbb C^n,m+1)$ as a $\operatorname{GL}_n(\mathbb{C})$ module (cf.\ \cite{l12}):
\[
\mathbb{TS}(\mathbb C^n,m+1)=\mathbb{C}^n\otimes \operatorname{S}^m(\mathbb{C}^n)= \operatorname{S}^{m+1} \mathbb{C}^n\oplus \operatorname{S}_{m,1}\mathbb{C}^n.
\]
In particular, when $n=2$ we have
\[
\mathbb{TS}(\mathbb C^2,m+1)= \operatorname{S}^{m+1} \mathbb{C}^2\oplus \operatorname{S}_{m,1}\mathbb{C}^2=\operatorname{S}^{m+1} \mathbb{C}^2 \oplus (\wedge^2\mathbb{C}^2\otimes \operatorname{S}^{m-1}\mathbb{C}^2).
\]
Tensors in $\operatorname{S}^{m+1}\mathbb{C}^2$ are just symmetric tensors, which  can be represented by $m+2$ parameters. 
More precisely, for each symmetric tensor $\mathcal T$, the homogeneous polynomial $\mathbf z^\mathsf{T}\big(\mathcal T\mathbf z^m\big)$ with $\mathbf z=(x,y)^\mathsf{T}$ can be parameterized as $F(x,y)=a_{m+1} x^{m+1}+ \cdots + a_0 y^{m+1}\in\mathbb C[x,y]$ for $a$'s.   
\begin{lemma}\label{eqn:Sm+1}
The system of eigenvalue equations associated to $\mathcal{T}$ is
\begin{align*}
 \frac{\partial F(x,y)}{\partial x} &=\lambda x^m,\\
 \frac{\partial F(x,y)}{\partial y} &=\lambda y^m.
\end{align*}
\end{lemma}

We characterize eigenvalues of a nonzero $\mathcal{T}\in \operatorname{S}^{m+1}\mathbb{C}^2$ in the next proposition.
\begin{proposition}\label{prop:eig-Sm+1}
Let $p_1,\cdots, p_k$ be distinct zeros of $F(x,y)$ in $\mathbb{P}^1$, with multiplicities $m_1,\cdots, m_k$ respectively. Let $L_{i}$ be the linear form vanishing on $p_i$ respectively. Eigenvalues of $\mathcal{T}$ are $0$ with multiplicity $\sum_{i=1}^{k}(m_i-1)$ and $\lambda_j= \frac{\partial F}{\partial x} (\alpha_j,\beta_j)/ \alpha_j^m , j=1,\dots, m+k-1$ with multiplicity one where $(\alpha_j,\beta_j)$ is a solution of 
\[
\frac{y^m \frac{\partial F}{\partial x}- x^m \frac{\partial F}{\partial y}}{\prod_{i=1}^k L_i^{m_i-1}}=0. 
\]
\end{proposition}
\begin{proof}
By the equations in Lemma~\ref{eqn:Sm+1} we obtain
\[
\lambda(y^m \frac{\partial F}{\partial x}- x^m \frac{\partial F}{\partial y})=0.
\]
Let $\lambda\ne 0$ be an eigenvalue of $\mathcal{T}$ and $(\alpha,\beta)\ne (0,0)$ be an eigenvector corresponding to $\lambda$. Then,  either $\frac{\partial F}{\partial x}(\alpha,\beta)\ne 0$ or $\frac{\partial F}{\partial y}(\alpha,\beta)\ne 0$, i.e., $(\alpha,\beta)$ is a root of 
\[
\frac{y^m \frac{\partial F}{\partial x}- x^m \frac{\partial F}{\partial y}}{\prod_{i=1}^k L_i^{m_i-1}}=0.
\]
Since $\mathcal{T}$ has $2m$ eigenvalues and $\sum_{i=1}^km_i=m+1$, we see that they are either $0$ or of the forms as claimed.
\end{proof}

To conclude this section, we consider eigenvalues of tensors in $\wedge^2\mathbb{C}^2\otimes \operatorname{S}^{m-1}\mathbb{C}^2$.
\begin{lemma}\label{lemma:Sm1equation}
The system of eigenvalue equations associated to a tensor $\mathcal{T}\in \wedge^2\mathbb{C}^2\otimes \operatorname{S}^{m-1}\mathbb{C}^2$ is of the form 
\begin{eqnarray*}
y f(x,y)=\lambda x^m,\\
-x f(x,y)=\lambda y^m,
\end{eqnarray*}
where $f(x,y)$ is a homogeneous polynomial of degree $m-1$.
\end{lemma}

\begin{proof}
Let us fix the standard basis $\{\mathbf e_1,\mathbf e_2\}$ for $\mathbb{C}^2$ then an element in $\wedge^2\mathbb{C}^2\otimes \operatorname{S}^{m-1}\mathbb{C}^2$ is 
\[
\mathcal{T}=(\mathbf e_1\wedge \mathbf e_2) \otimes f,
\]
where $f\in \operatorname{S}^{m-1}\mathbb{C}^2$. We identify $\operatorname{S}^{m-1}\mathbb{C}^2$ with the space of homogeneous polynomials of degree $m-1$ in two variables with coefficients in the field of complex numbers.
Expand $\mathbf e_1\wedge \mathbf e_2$ and write out the equation system corresponding to $\mathcal{T}$, the claim follows.
\end{proof}

Let $\mathcal{T}\in \wedge^2\mathbb{C}^2\otimes \operatorname{S}^{m-1}\mathbb{C}^2$, then we describe eigenvalues of $\mathcal{T}$ in the next proposition.
\begin{proposition}\label{prop:Sm1}
Eigenvalues of $\mathcal{T}$ are $0$ with multiplicity $m-1$ and $\omega_if(1,\omega_i),i=0,\dots, m$ where $\omega_i$ is an $(m+1)$-th root of $-1$.
\end{proposition}

\begin{proof}
If $\lambda\ne 0$ is an eigenvalue of $\mathcal{T}$ then 
\begin{align*}
yf(x,y)=\lambda x^m \textit{ and }
-xf(x,y)=\lambda y^m.
\end{align*}
Thus 
\[
\lambda (x^{m+1}+y^{m+1}) =0.
\]
Since $\lambda\ne 0$ we can derive 
\[
y=\omega_i x, i=0,\dots, m. 
\]
It is easy to obtain 
\[
\lambda=y f(x,y)/x^m = \omega_i f(1,\omega_i).
\]
Lastly, every homogeneous polynomial $f(x,y)$ definitely has a nontrivial solution in $\mathbb C^2$ by Hilbert's zero theorem, we conclude that $0$ is also an eigenvalue of $\mathcal T$. Since the total number of eigenvalues is $2m$, we see that $0$ gets the rest multiplicity $m-1$.
\end{proof}

\begin{remark}
We notice that Proposition \ref{prop:Sm1} gives an algorithm to reconstruct a tensor $\mathcal{T}\in \wedge^2\mathbb{C}^2\otimes\operatorname{S}^{m-1}\mathbb{C}^2$ from given $m+1$ numbers $\lambda_0,\dots, \lambda_{m}$ such that eigenvalues of $\mathcal{T}$ are $\lambda_0,\dots, \lambda_{m}$ and zero by solving a linear system. Namely, we consider the following linear system
\[
\left(
\begin{matrix}
1 & 1 & \cdots & 1 & 1\\
\omega_1 & \omega_1^2 & \cdots & \omega_1^{m-1} & \omega_1^m \\
\vdots & \vdots & \ddots & \vdots & \vdots \\
\omega_m & \omega_m^2 & \cdots & \omega_m^{m-1} & \omega_m^m  
\end{matrix}
\right)\left(
\begin{matrix}
a_0 \\
a_1 \\
\vdots \\
a_{m-1}
\end{matrix}
\right)=\left(
\begin{matrix}
\lambda_0 \\
\lambda_1 \\
\vdots \\
\lambda_{m}
\end{matrix}
\right),
\]
where $\omega_i$'s are the $m+1$-th roots of $-1$.
\begin{enumerate}
\item If this overdetermined linear system has no solution then there is no tensor $\mathcal{T}\in \wedge^2\mathbb{C}^2\otimes\operatorname{S}^{m-1}\mathbb{C}^2$ can have $\{\lambda_0,\dots, \lambda_{m},0,\dots, 0\}$ as the multiset of eigenvalues.
\item If this overdetermined linear system has a solution then the solution gives a homogeneous polynomial $f(x,y)=\sum_{i=0}^{m-1}a_ix^{m-1-i}y_i$ which gives the desired tensor $\mathcal{T}$ (cf.\ Lemma~\ref{lemma:Sm1equation}). 
\end{enumerate}  
\end{remark}

\section{The Exceptional Cases}\label{sec:general}
In this section, we show that the eigenvalue map $\phi : \mathbb{TS}(\mathbb C^n,m+1)\rightarrow\mathbb C^{nm^{n-1}}/\mathfrak{S}(nm^{n-1})$ is dominant for the \textit{exceptional cases}
\[
(n,m)=(3,2),\ (4,2),\ (3,3).
\]
We use Propositions~\ref{prop:differential} and \ref{prop:tensor} to prove the results. The basic idea is the same as Section~\ref{sec:two}: finding a point in $\mathbb{TS}(\mathbb C^n,m+1)$ such that the differential of the coefficient map $\mathbf c :\mathbb{TS}(\mathbb C^n,m+1)\rightarrow\mathbb C^{nm^{n-1}}$ at this point has the maximal rank $nm^{n-1}$. The difference is instead of proving a generic property on the Jacobian matrix, we find a concrete point at which the Jacobian matrix is of full rank. 
\subsection{Macaulay's formulae of characteristic polynomials}\label{sec:mac}
We refer to \cite{clo98,s02,gkz94,h77} for the theory and computation of resultants and hyperdeterminants. The determinant of a tensor
is actually the resultant of a specially constructed system of homogeneous polynomials of the same degree \cite{hhlq13}. 

Let
\[
d=nm-n+1,
\]
and 
\[
S=\{x_1^d,x_1^{d-1}x_2,\dots,x_n^d\}
\]
be the set of monomials in $x_1,\dots,x_n$ of degree $d$ in lexicographic order. A monomial of degree $d$ is written as $\mathbf x^{\alpha}=x_1^{\alpha_1}\dots x_n^{\alpha_n}$ with $\alpha\in\mathbb N^n$ and $\alpha_1+\dots+\alpha_n=d$. 
The set $S$ divides into $n$ subsets as follows:
\[
S_i:=\{\mathbf x^\alpha \in S : \alpha_i\geq m,\ \alpha_j<m\ \text{for all }j=1,\dots,i-1\},\ \text{for all }i=1,\dots,n.
\]
It is easy to see that $\{S_1,\dots,S_n\}$ are mutually disjoint and $\cup_{i=1}^nS_i=S$. Note that the cardinality of $S$ is
\[
w=|S|={d+n-1\choose d}.
\]

Let $\mathcal T\in\mathbb{TS}(\mathbb C^n,m+1)$. We write
\[
f_i(\mathbf x):=(\mathcal T\mathbf x^m-\lambda\mathbf x^{[m]})_i
\]
as the $i$th defining equation for the eigenvalue problem for $i=1,\dots,n$.
For the $n$ homogeneous polynomials $f_1(\mathbf x),\dots,f_n(\mathbf x)$ in $n$ variables $\mathbf x=(x_1,\dots,x_n)$, parameterized by $\mathcal T$ and $\lambda$, we can formulate a system of $w$ homogeneous polynomials
\begin{equation}\label{mac:matrix}
\mathbf x^{\alpha-m\mathbf e_i}\cdot f_i(\mathbf x),\ \text{for all }\mathbf x^\alpha\in S_i,\ \text{for all }i=1,\dots,n,
\end{equation}
where $\mathbf e_i\in\mathbb R^n$ is the $i$th standard basis vector. This system of polynomials is naturally indexed by monomials $\mathbf x^\alpha\in S$. 
With respect to the basis $S$, we can represent the system \eqref{mac:matrix} as a matrix $R\in\mathbb C[\mathcal T,\lambda]^{w\times w}$. 
A monomial $\mathbf x^\alpha$ is reduced, if there exists exactly one $i\in\{1,\dots,n\}$ such that $\alpha_i\geq m$. The submatrix of $R$ obtained by deleting all rows and columns of reduced monomials is denoted by $R'$. Note that the entries of both $R$ and $R'$ are 
linear forms of the variables $t_{ii_1\dots i_m}$ and $\lambda$. 

It follows from Macaulay's formula for resultant  (cf.\ \cite{m02}) that the characteristic polynomial of $\mathcal T$ is
\begin{equation}\label{mac:characteristic}
\operatorname{det}(\mathcal T-\lambda\mathcal I)=\pm \frac{\operatorname{det}(R)}{\operatorname{det}(R')}. 
\end{equation}
 
With the characteristic polynomial \eqref{mac:characteristic}, we can compute out the coefficient map $\mathbf c : \mathbb{TS}(\mathbb C^n,m+1)\rightarrow\mathbb C^{nm^{n-1}}$ and its Jacobian matrix $H$. Note that, we may restrict our map $\mathbf c$ on a subspace $V\subseteq\mathbb{TS}(\mathbb C^n,m+1)$ as long as the dimension of $V$ is larger than $nm^{n-1}$ (cf.\ Proposition~\ref{prop:necessary-general}) to reduce the computational cost.

\subsection{ Third order three dimensional tensors}\label{sec:32}
In this section, we present the detailed computation for third order three dimensional tensors, i.e., $(n,m)=(3,2)$. The details serve as an example to illustrate the method in Section~\ref{sec:mac}. 
The computation has been majorly conducted by Macaulay2 \cite{gs} together with Matlab. 

For any tensor $\mathcal T\in\mathbb {TS}(\mathbb C^3,3)$, its system of eigenvalue equations is
\[
\begin{cases}\sum_{j,k=1}^3t_{1jk}x_jx_k=\lambda x_1^2,\\
\sum_{j,k=1}^3t_{2jk}x_jx_k=\lambda x_2^2,\\ 
\sum_{j,k=1}^3t_{3jk}x_jx_k=\lambda x_3^2,
\end{cases}
\]
which can be equivalently parameterized as 
\[
\begin{cases}f_1(A,\lambda,\mathbf x):=a_{11}x_1^2+a_{12}x_1x_2+a_{13}x_2^2+a_{14}x_1x_3+a_{15}x_2x_3+a_{16}x_3^2-\lambda x_1^2=0,\\
f_2(A,\lambda,\mathbf x):=a_{21}x_1^2+a_{22}x_1x_2+a_{23}x_2^2+a_{24}x_1x_3+a_{25}x_2x_3+a_{26}x_3^2-\lambda x_1^2=0,\\
f_3(A,\lambda,\mathbf x):=a_{31}x_1^2+a_{32}x_1x_2+a_{33}x_2^2+a_{34}x_1x_3+a_{35}x_2x_3+a_{36}x_3^2-\lambda x_1^2=0.
\end{cases}
\]

Let
\[
S=\{x_1^4,x_1^3x_2,\dots,x_2^4,\dots,x_3^4\}
\]
be the set of all monomials of $x_1, x_2, x_3$ with total degree $4$ in lexicographic order, and
\[
T_1=\{x_1^2, x_1x_2, x_1x_3, x_2^2, x_2x_3, x_3^2\},\
T_2=\{x_1x_2, x_1x_3, x_2^2, x_2x_3, x_3^2\},
\]
and
\[
T_3=\{x_1x_2, x_1x_3,x_2x_3,x_3^2\}.
\]
It follows that the cardinality of $S$ is
\[
|S|={3+4-1\choose 4}=15.
\]

We generate a system of $15$ polynomial equations via
\[
f_ig^i_j = 0,\ \text{for all }g^i_j\in T_i,\ \text{for all }i=1,2,3. 
\]
Regarding $f_ig^i_j\in\mathbb C[A,\lambda][\mathbf x]$, we can get a square matrix $M\in\big(\mathbb C[A,\lambda]\big)^{15\times 15}$ as the coefficient matrix of the polynomial equations 
\[
f_1g^1_1=0,\dots,f_1g^1_6=0, f_2g^2_1=0,\dots,f_2g^2_5=0, f_3g^3_1=0,\dots,f_3g^3_4=0
\]
in the canonical basis $S$. It follows from Section~\ref{sec:mac} that the characteristic polynomial of $\mathcal T$ is
\[
\operatorname{det}(\mathcal T-\lambda\mathcal I)=\pm \frac{\operatorname{det}(M)}{(a_{11}-\lambda)^2(a_{23}-\lambda)}. 
\]
The matrix $M$ is
\[
\small
\begin{bmatrix} 
a_{11}& a_{12}& a_{13}& a_{14}& a_{15}& a_{16}&    0&    0&    0&    0&    0&    0&    0&    0&    0\\
   0& a_{11} & a_{12}&    0& a_{14}&    0& a_{13}& a_{15}&    0&    0&    0& a_{16}&    0&    0&    0\\
    0&    0& a_{11} &    0&    0&    0& a_{12}& a_{14}& a_{13}& a_{15}& a_{16}&    0&    0&    0&    0\\
     0&    0&    0& a_{11} & a_{12}& a_{14}&    0& a_{13}&    0&    0&    0& a_{15}& a_{16}&    0&    0\\
   0&    0&    0&    0& a_{11} &    0&    0& a_{12}&    0& a_{13}& a_{15}& a_{14}&    0& a_{16}&    0\\
     0&    0&    0&    0&    0& a_{11} &    0&    0&    0&    0& a_{13}& a_{12}& a_{14}& a_{15}& a_{16}\\
    0& a_{21}& a_{22}&    0& a_{24}&    0& a_{23} & a_{25}&    0&    0&    0& a_{26}&    0&    0&    0\\
    0&    0&    0& a_{21}& a_{22}& a_{24}&    0& a_{23} &    0&    0&    0& a_{25}& a_{26}&    0&    0\\
    0&    0& a_{21}&    0&    0&    0& a_{22}& a_{24}& a_{23} & a_{25}& a_{26}&    0&    0&    0&    0\\
    0&    0&    0&    0& a_{21}&    0&    0& a_{22}&    0& a_{23} & a_{25}& a_{24}&    0& a_{26}&    0\\
    0&    0&    0&    0&    0& a_{21}&    0&    0&    0&    0& a_{23} & a_{22}& a_{24}& a_{25}& a_{26}\\
    0& a_{31}& a_{32}&    0& a_{34}&    0& a_{33}& a_{35}&    0&    0&    0& a_{36}&    0&    0&    0\\
    0&    0&    0& a_{31}& a_{32}& a_{34}&    0& a_{33}&    0&    0&    0& a_{35}& a_{36}&    0&    0\\
    0&    0&    0&    0& a_{31}&    0&    0& a_{32}&    0& a_{33}& a_{35}& a_{34}&    0& a_{36}&    0\\
   0&    0&    0&    0&    0& a_{31}&    0&    0&    0&    0& a_{33}& a_{32}& a_{34}& a_{35}& a_{36}
\end{bmatrix}-\lambda I,
\]
where $I$ is the identity matrix of appropriate size. 

If we restrict the tensor space to be with $a_{21}=a_{31}=a_{13}=a_{33}=0$, then we have 
\[
\operatorname{det}(\mathcal T-\lambda\mathcal I)=\frac{\operatorname{det}(M)}{(a_{11}-\lambda)^2(a_{23}-\lambda)}=\operatorname{det}(M')
\]
with
\[
M'=\small\begin{bmatrix} 
    a_{11}&    0&    0&    0& a_{12}& a_{14}& a_{15}& a_{16}&    0&    0&    0&    0\\
     0& a_{11}& a_{12}& a_{14}&    0&0&      0&    0& a_{15}& a_{16}&    0&    0\\
   0&    0& a_{11}&    0&    0& a_{12}&     0& a_{15}& a_{14}&    0& a_{16}&    0\\
    0&    0&    0& a_{11}&    0&    0&    0&    0& a_{12}& a_{14}& a_{15}& a_{16}\\
  a_{22}&    0& a_{24}&    0& a_{23}& a_{25}&        0&    0& a_{26}&    0&    0&    0\\
  0& 0& a_{22}& a_{24}&    0& a_{23}&       0&    0& a_{25}& a_{26}&    0&    0\\
      0&    0&0&    0&    0& a_{22}&     a_{23}& a_{25}& a_{24}&    0& a_{26}&    0\\
      0&    0&    0& 0&    0&    0&    0&     a_{23}& a_{22}& a_{24}& a_{25}& a_{26}\\
    a_{32}&    0& a_{34}&    0&0& a_{35}&        0&    0& a_{36}&    0&    0&    0\\
     0& 0& a_{32}& a_{34}&    0& 0&        0&    0& a_{35}& a_{36}&    0&    0\\
      0&    0& 0&    0&    0& a_{32}&     0& a_{35}& a_{34}&    0& a_{36}&    0\\
   0&    0&    0& 0&    0&    0&    0&     0& a_{32}& a_{34}& a_{35}& a_{36}
\end{bmatrix}-\lambda I.
\]
Note that we restricted our coefficient map $\mathbf c$ on a linear subspace $V$ of dimension $14$. We use Macaulay2 to compute the $12\times 14$ Jacobian matrix. The evaluation of this matrix at the point
\begin{align*}
a_{11}&=1, a_{12}=2, a_{14}=3, a_{15}=4, a_{16}=5,\\
a_{22}&=6, a_{23}=7, a_{24}=8,a_{25}=9,a_{26}=10,\\
a_{32}&=11,a_{34}=12,a_{35}=13,a_{36}=14
\end{align*}

is 
\begin{multline*}
\tiny\left[\begin{matrix}   -4       &   0     &      0       &     0     &      0    &       0     &  -4     \\   348   &       -12     &    -24   &       0    &       0  &         -4 &       
      324        \\
      -11948    &   528   &      1575   &      -336   &     -420   &     123       &
      -10190      \\ 229449      & -6573  &     -42450 &      9606   &     14460  &     -435      &  178549  \\
       -2841839   &  8007   &     669288  &     -123924  &   -254385 &    -40559    &
      -1983021\\
       24015886  &   693225  &    -6820251 &    716979 &     2947131 &    838271    &
      14685692 \\
       -141005226 &  -8897502 &   46520475  &   -55500  &    -23636976  & -7517499  &
      -74200394  \\
       577067743 &   52779339 &   -214721160 &  -19191762 &  128734014  & 37178105  &
      258147039 \\
       -1615274021&  -168212115&  650003046&    85305210 &   -443297901 & -110199791&
      -603940123   \\
       2874026450 &  286931673 &  -1165235823&  -145117011&  852500403 &  194375103 &
      876534196  \\
       -2794768018&  -259796358&  1046384496 &  111968790 &  -778096974&  -179135234&
      -672256468  \\
       1066887388&   96499788  &  -356759172  & -33512052 &  261090648 &  64501920  &
      202844400  
      \end{matrix}\right.\\
      \tiny\left. \begin{matrix}       0     &      0   &        0      &     0  &    0&
        0   &        -4 \\         0     &      -26   &      0   &        0   &        -6        &
      -18  &       296        \\  -158  &      1685 &       -382 &       -223&        156       &
                 652    &     -8362      \\ 
                     4943 &       -43158 &     16014  &     5650  &      -1511     &  -6084  &     129887 \\
           -54113&      607211 &     -282327 &    -46378 &     -27079    &
      -23151 &     -1258172 \\   150466 &     -5348868  &  2845660  &   -150997 &    963371    &1047467  &   7850991 \\
         170837&     31218785 &   -18669471&   5370313 &    -11890233 &  -11634444 &  -30119950 \\
         -16953189&   -122285430&  82279634 &   -38993472&   78057193  &
      68998306&    58177003\\
         64320415&    313612725 &  -232274827 & 133470656 &  -288207851&
      -225232919&  -1425840   \\
       -125125090&  -487990298&  384765126  & -232073969 & 569148965 &
      386696721 &  -187322475 \\
        121312760 &  392623914&   -320537057&  201649792  & -531624753&
      -319797178&  252532328  \\
        -45364410 &  -122396540 & 101857630  & -69231372&   183581748 & 99950648 &   -98555702
      \end{matrix}\right]
\end{multline*}
Using either Matlab or Macaulay2, we can check that the above matrix has full rank $12$. 
Therefore, we arrive at the next proposition. 
\begin{proposition}\label{prop:dominant-33}
The eigenvalue map $\phi : \mathbb{TS}(\mathbb C^3,3)\rightarrow\mathbb C^{12}/\mathfrak{S}(12)$ is dominant. 
\end{proposition}

\subsection{Fourth order three dimensional and third order four dimensional tensors}\label{sec:42-33} 
In this section, we show that the eigenvalue maps $\phi : \mathbb{TS}(\mathbb C^3,4)\rightarrow\mathbb C^{27}/\mathfrak{S}(27)$ and $\phi : \mathbb{TS}(\mathbb C^4,3)\rightarrow\mathbb C^{32}/\mathfrak{S}(32)$ are both dominant. Note that, symbolically, the determinants of tensors of both formats $\mathbb{TS}(\mathbb C^3,4)$ and $ \mathbb{TS}(\mathbb C^4,3)$ are mostly likely to have millions of terms, regarding the relationship between determinants and hyperdeterminnats (cf.\ \cite{o12}) and already the approximately 3 million terms for the hyperdeterminant of tensors in $\mathbb T(\mathbb C^2,4)$ (cf.\ \cite{hsyy08}). It would be impractical or impossible using Macaulay2 to compute out the characteristic polynomial $\operatorname{det}(\mathcal T-\lambda \mathcal I)$ for a symbolic tensor in these two cases. 

For a map $\mathbf f : V\rightarrow W$ between two vector spaces $V$ and $W$, whenever the differential $d_{\mathbf x}\mathbf f$
exists at a point $\mathbf x$, we have that the directional derivative of $\mathbf f$ at direction $\mathbf y\in V$ is
\begin{equation}\label{directional}
\mathbf f^\mathsf{\prime}(\mathbf x;\mathbf y)=(d_{\mathbf x}\mathbf f) \mathbf y.
\end{equation}
Let $\operatorname{dim}(V)\geq \operatorname{dim}(W)$. 
As a linear map from $V$ to $W$, $d_{\mathbf x}\mathbf f$ is of maximal rank $\operatorname{dim}(W)$ if we can find a set of directions $\{\mathbf y_1,\dots,\mathbf y_k\}\subset V$ with $k\geq \operatorname{dim}(W)$ such that
\[
\operatorname{rank}\big([(d_{\mathbf x}\mathbf f) \mathbf y_1,\ \dots,\ (d_{\mathbf x}\mathbf f) \mathbf y_k]\big)=\operatorname{dim}(W).
\]
Note that here our map is the coefficient map $\mathbf c$ (cf.\ \eqref{coefficient-map}). $V$ is either $\mathbb{TS}(\mathbb C^3,4)$ or $ \mathbb{TS}(\mathbb C^4,3)$, and $W$ is respectively either $\mathbb C^{27}$ or $\mathbb C^{32}$. In both cases, we choose $k=\operatorname{dim}(V)$. 
We use formula \eqref{directional} to compute $(d_{\mathbf x}\mathbf f) \mathbf y_i$ for each $i=1,\dots,k$. 
We first choose a point $\mathcal T\in V$ and a set of directions $\{\mathcal T_1,\dots,\mathcal T_k\}$, which will be chosen as the set of standard basis of the space $V$. 
Then, we compute the characteristic polynomial of  $\mathcal T + t\mathcal T_i$ with parameter $t$
\[
\operatorname{det}(\mathcal T+t\mathcal T_i-\lambda\mathcal I)
\]
for all $i=1,\dots,k$. Note that we have only two symbolic variables $\lambda$ and $t$ now. 
Write $\operatorname{det}(\mathcal T+t\mathcal T_i-\lambda\mathcal I)$ as
\[
\operatorname{det}(\mathcal T+t\mathcal T_i-\lambda\mathcal I)=\sum_{s=0}^Nc_s(t)\lambda^s,
\]
for appropriate $N$, which is either $27$ or $32$. Note that $c_N(t)=\pm 1$. 
It follows that
\[
\mathbf c^\mathsf{\prime}(\mathcal T;\mathcal T_i)=(d_{\mathcal T}\mathbf c) \mathcal T_i=(c_{N-1}^\mathsf{\prime}(0),\dots,c_0^\mathsf{\prime}(0))^\mathsf{T}. 
\]
In this way, we can try to find a tensor $\mathcal T\in V$ such that the resulting matrix
\[
[(d_{\mathcal T}\mathbf c) \mathcal T_1,\ \dots, (d_{\mathcal T}\mathbf c) \mathcal T_k]
\]
has full rank. In fact, if such a tensor $\mathcal{T}$ exists then a generic tensor will work. For $V=\mathbb{TS}(\mathbb C^3,4)$, the differential of the coefficient map at the tensor point (only independent entries are listed)
\begin{align*}
&t_{1111}=1,\ t_{1112}=-1/3,\ t_{1122}=2/3,\ t_{1222}=-2,\
t_{1113}=1,\ t_{1123}=-1/2,\ t_{1223}=4/3,\\ 
&t_{1133}=-4/3,\
t_{1233}=5/3,\ t_{1333}=-5,
t_{2111}=6,\ t_{2112}=-2,\ t_{2122}=7/3,\ t_{2222}=-7,\
t_{2113}=8/3,\\ 
&t_{2123}=-4/3,\ t_{2223}=3,\ t_{2133}=-3,\
t_{2233}=1/3,\ t_{2333}=2,\
t_{3111}=3,\ t_{3112}=4/3,\ t_{3122}=5/3,\\ 
&t_{3222}=6,\
t_{3113}=0, t_{3123}=-1/6,\ t_{3223}=-2/3,\ t_{3133}=-1,\
t_{3233}=-4/3,\ t_{3333}-5
\end{align*}
is of full rank $27$; 
and for $V=\mathbb{TS}(\mathbb C^4,3)$, the differential of the coefficient map at the tensor point (again, only independent entries are listed)
\begin{align*}
&t_{111}=1,\ t_{112}=-1/2,\ t_{122}=2,\ t_{113}=-1,\ t_{123}=3/2,\ t_{133}=-3,
t_{114}=2,\ t_{124}=-2,\ t_{134}=5/2,\\ 
&t_{144}=-5,
t_{211}=6,\ t_{212}=-3,\ t_{222}=7,\ t_{213}=-7/2,\ t_{223}=4,\ t_{233}=-8,
t_{214}=9/2,\ t_{224}=-9/2,\\ 
&t_{234}=1/2,\ t_{244}=2,
t_{311}=3,\ t_{312}=2,\ t_{322}=5,\ t_{313}=3,\ t_{323}=0,\ t_{333}=-1,
t_{314}=-1,\\ 
& t_{324}=-3/2,\ t_{334}=-2,\ t_{344}-5,
t_{411}=1,\ t_{412}=1,\ t_{422}=3/2,\ t_{413}=2,\ t_{423}=5/2,\ t_{433}=6,\\
&t_{414}=7/2,\ t_{424}=4,\ t_{434}=9/2,\ t_{444}=10
\end{align*}
is of full rank $32$.

Therefore, we have the next result from Propositions~\ref{prop:differential} and \ref{prop:tensor}. 
\begin{proposition}\label{prop:dominant-34-42}
The eigenvalue maps $\phi : \mathbb{TS}(\mathbb C^3,4)\rightarrow\mathbb C^{27}/\mathfrak{S}(27)$ and
$\phi : \mathbb{TS}(\mathbb C^4,3)\rightarrow\mathbb C^{32}/\mathfrak{S}(32)$
are both dominant. 
\end{proposition}


\section{Final remarks}\label{sec:remarks}
\subsection{Reconstruction of a tensor}
The understanding of the ranges of the multiset-valued maps would be an essential step to understand the configurations of the eigenvalues of general tensors \footnote{There is a parallel research on another concept of eigenvectors \cite{ass}.}.
\begin{proposition}\label{prop:sub}
For any integers $m,n\geq 2$, there is a set $W\subset\mathbb C^{nm^{n-1}}/\mathfrak{S}(nm^{n-1})$ whose closure has dimension $2\lfloor \frac{n}{2}\rfloor m$ contained in the closure of the image $\phi(\mathbb{T}(\mathbb C^n,m+1))$. 
\end{proposition}
\begin{proof}
For any $n\geq 2$, we can take $\lfloor \frac{n}{2}\rfloor$ tensors $\mathcal A_i\in \mathbb{T}(\mathbb C^2,m+1)$, and possibly a scalar $\alpha$ (when $n$ is odd) as subtensors to form a diagonal block tensor $\mathcal T\in\mathbb{T}(\mathbb C^n,m+1)$. It follows from \cite{hhlq13} that
\[
\operatorname{det}(\mathcal T-\lambda\mathcal I)=(\alpha-\lambda)^q\prod_{i=1}^{\lfloor \frac{n}{2}\rfloor}\big[\operatorname{det}(\mathcal A_i-\lambda\mathcal I)\big]^p
\]
for some positive integers $p,q$. So, it is clear that the set $\phi(\mathbb{T}(\mathbb C^2,m+1))\times \dots\times \phi(\mathbb{T}(\mathbb C^2,m+1))$ with $\lfloor \frac{n}{2}\rfloor$ copies can be embedded into $\phi(\mathbb{T}(\mathbb C^n,m+1))$. It then follows from Proposition~\ref{prop:dominant-n2} that the closure of this set has dimension $2\lfloor \frac{n}{2}\rfloor m$. 
\end{proof}
Similarly, we can use Propositions~\ref{prop:dominant-33} and \ref{prop:dominant-34-42} to refine the blocks to get a variety with larger dimension in some cases. However, in general it is far away from the following expected dimension \eqref{eqn:expect}.  
\subsection{The dimension of the image of $\phi$}
It follows from Theorem~\ref{theorem:dominant} that for most tensor spaces $\mathbb T(\mathbb C^n,m+1)$, the multiset-valued eigenvalue map is not dominant. Thus, it is reasonable to expect that the dimension of $\overline{\phi(\mathbb{T}(\mathbb C^n,m+1))}$  is 
\begin{equation}\label{eqn:expect}
\min\bigg\{n{n+m-1\choose m}, nm^{n-1}\bigg\},
\end{equation}
since $\phi(\mathbb{T}(\mathbb C^n,m+1))=\phi(\mathbb{TS}(\mathbb C^n,m+1))$ and $\operatorname{dim}(\mathbb{TS}(\mathbb C^n,m+1))=n{n+m-1\choose m}$.
We also want to point out that \eqref{eqn:expect} may not hold for all $m,n\geq 2$. We tested, by a similar method as Section~\ref{sec:42-33}, the case $(n,m)=(3,4)$. Note that the tensor space is of dimension $45$ while the number of eigenvalues is $48$. 
However, the ranks of the resulting Jacobian matrices for both the following two points (only independent elements are listed) are of the same value $43$. 
\begin{align*}
&t_{11111}=0,\ t_{11112}=3/4,\ t_{11122}=5/6,\ t_{11222}=1/4,\ t_{12222}=0,\ t_{11113}=-5/4,
t_{11123}=-1/6,\\
& t_{11223}=-1/6,\ t_{12223}=5/4,\
t_{11133}=-1/2,\ t_{11233}=-1/3,\ t_{12233}=-1/6,\ t_{11333}=-1,\\
& t_{12333}=1/2,
 t_{13333}=4,\ t_{21111}=3,\ t_{21112}=5/4,\ t_{21122}=-1/2,\ t_{21222}=1,\ t_{22222}=-2,\\
 & t_{21113}=1,\ t_{21123}=1/6,\ t_{21223}=-5/12,\ t_{22223}=-1,\
t_{21133}=-1/6,\ t_{21233}=0,\ t_{22233}=-1/3,\\
& t_{21333}=1,\ t_{22333}=-5/4,\
 t_{23333}=-1,
t_{31111}=-3,\ t_{31112}=-5/4,\ t_{31122}=-1/3,\\
& t_{31222}=-1/2,\
 t_{32222}=0,\ t_{31113}=3/4,\
t_{31123}=1/6,\
 t_{31223}=1/3,\  t_{32223}=2/3,\ 
t_{31133}=-2/3,\\
& t_{31233}=-1/6,\
 t_{32233}=1/3,\ t_{31333}=1,\ t_{32333}=-1,\ t_{33333}=0, 
\end{align*}
and
\begin{align*}
&t_{11111}=7,\ t_{11112}=-3/2,\ t_{11122}=-4/3,\ t_{11222}=-9/4,\ t_{12222}=8,\ t_{11113}=7/4,
t_{11123}=3/4,\\
& t_{11223}=7/12,\ t_{12223}=-7/6,\ 
t_{11133}=5/6,\ t_{11233}=-1/12,\ t_{12233}=5/6,\ t_{11333}=1,\ t_{12333}=9/4,\\ 
& t_{13333}=0,\ t_{21111}=10,\ t_{21112}=-1,\ t_{21122}=0,\ t_{21222}=5/4,\ t_{22222}=-1,\ t_{21113}=7/4,\\
&
t_{21123}=7/12,\
 t_{21223}=0,\ t_{22223}=-7/6,\ 
t_{21133}=-7/6,\ t_{21233}=-1/4,\ t_{22233}=5/6,\\
& t_{21333}=-5/2,\
 t_{22333}=1,\
 t_{23333}=6,\ t_{31111}=8,\ t_{31112}=-3/2,\ t_{31122}=-1/3,\ t_{31222}=-5/4,\\
 & t_{32222}=4,\
 t_{31113}=9/4,\
t_{31123}=3/4,\ t_{31223}=-1/3,\ t_{32223}=-4/3,\ 
t_{31133}=1/6,\\ & t_{31233}=5/6,\ t_{32233}=-1,\ t_{31333}=3/2,\ t_{32333}=5/2,\ t_{33333}=6.  
\end{align*}
Nevertheless, we have considerable confidence to believe that \eqref{eqn:expect} is true for all but a finite exceptions, as such phenomena happen in tensor problems, e.g., the famous Alexander-Hirschowitz theorem for symmetric tensor rank \cite{ah95}.   
\subsection{The closedness of the image of $\phi$} We proved in Proposition \ref{prop:dominant-n2} that for a generic multiset of total multiplicity $2m$ of complex numbers, the inverse eigenvalue problem is solvable. A natural question to ask is: is the inverse eigenvalue problem solvable for any multiset of total multiplicity $2m$ of complex numbers, i.e., is $\phi$ surjective? It is easy to show that when $m=2$, the answer is affirmative. To this end, we will need the following result.
\begin{proposition}[\cite{f77}]\label{thm:friedland}
Let $f=(f_1,\dots,f_n):\mathbb{C}^n\to \mathbb{C}^n$ be a polynomial map where each $f_i$ is homogeneous. If $f(x_1,\dots, x_n)=0$ has only a trivial solution $(0,\dots,0)$ then $f(x_1,\dots,x_n)=\omega$ is always solvable for any $\omega\in \mathbb{C}^n$.
\end{proposition}
\begin{proposition}[Cubic plane tensor]\label{prop:closedness m=2}
Given any multiset $S$ of total multiplicity four of complex numbers, there exists a tensor $\mathcal{T}$ in $\mathbb{T}(\mathbb{C}^2,3)$ such that the set of eigenvalues of $\mathcal{T}$ is $S$.
\end{proposition}
\begin{proof}
As before, we use $a_0,a_1,a_2,b_0,b_1,b_2$ to parametrize $\mathbb{TS}(\mathbb C^2,3)$ which is isomorphism to $\mathbb{C}^6$. Let $\mathbf{c}: \mathbb{C}^6\to \mathbb{C}^4$ be the map sending the vector $(a_0,a_1,a_2,b_0,b_1,b_2)$ to $(c_1,c_2,c_3,c_4)$ where $c_i$ is the codegree $i$ coefficient of the characteristic polynomial of the tensor determined by $(a_0,a_1,a_2,b_0,b_1,b_2)$, $i=1,\dots,4$. Hence $c_i$'s are homogeneous polynomials of degree $(4-i)$ respectively. By Proposition~\ref{thm:friedland} and Proposition~\ref{prop:tensor}, it is sufficent to find a four dimension linear subspace $L\subset \mathbb{C}^6$ such that $\mathbf c^{-1}((0,0,0,0))\cap L$ is $(0,0,0,0)$. We consider the linear subspace $L$ defined by equations 
\[
a_1+b_1+b_2 =0, a_2+b_0 =0.
\]
Then it is easy to verify that $L\cap \mathbf c^{-1}((0,0,0,0))$ is $(0,0,0,0)$.
\end{proof}
\begin{remark}
We use Macaulay2 to see that $L\cap \mathbf c^{-1}((0,0,0,0))$ is $(0,0,0,0)$. Since $n$ generic homogeneous polynomials only have a trivial solution, the existence of $L$ in the proof of Proposition \ref{prop:closedness m=2} implies that a generic four dimensional subspace of $\mathbb{C}^6$ should work. 
\end{remark}
\begin{remark}
It is tempting to extend the proof of Proposition \ref{prop:closedness m=2} to show that $\mathbf c$ is surjective in general. However, on one hand it is difficult to compute the intersection of $\mathbf c^{-1}(\mathbf 0)$ with a generic linear space of dimension $2m$ in general. On the other hand, when $m=3$ the dimension of $\mathbf c^{-1}(\mathbf 0)$ is three which is larger than the expected dimension two, hence the method used for $m=2$ does not work for $m=3$.
\end{remark}
\subsection*{Acknowledgement} 
This work is partially supported by
National Science Foundation of China (Grant No. 11401428). 
\bibliographystyle{model6-names}

\end{document}